\documentclass[12pt]{amsart}
\let\oldtocsection=\tocsection

\let\oldtocsubsection=\tocsubsection

\let\oldtocsubsubsection=\tocsubsubsection

\renewcommand{\tocsection}[2]{\hspace{0em}\oldtocsection{#1}{#2}}
\renewcommand{\tocsubsection}[2]{\hspace{1em}\oldtocsubsection{#1}{#2}}
\renewcommand{\tocsubsubsection}[2]{\hspace{2em}\oldtocsubsubsection{#1}{#2}}

\pdfoutput=1
\usepackage{hyperref}
\hypersetup{
  pdfinfo={
    Title={Imploded Cross-Sections},
    Author={LISA JEFFREY AND SINA ZABANFAHM},
    Subject={},
    Keywords={}
  }
}
    \usepackage{amssymb,amsmath,latexsym, amscd, graphicx, url}
\usepackage{amsmath}
\usepackage{amssymb}
\usepackage{tipa}
\usepackage{tikz}
\usepackage{tikz-cd}
\usepackage{centernot}
\usepackage{mathtools}
\usepackage{stmaryrd}
\usepackage{ mathrsfs}
\usepackage{fullpage}
\usepackage{hyperref}
\usepackage{cancel}
\usepackage{leftidx}
\usepackage{ upgreek}
\usetikzlibrary{matrix,arrows}
\hypersetup{colorlinks=true, linktoc=all, linkcolor=blue}

\newtheorem{theorem}{Theorem}[section]
\newtheorem{lemma}[theorem]{Lemma}

\newtheorem{definition}[theorem]{Definition}

\newtheorem{corollary}[theorem]{Corollary}
\newtheorem{rk}[theorem]{Remark}

    \theoremstyle{definition}
    \newtheorem{remark}[theorem]{Remark}
    \newtheorem{example}[theorem]{Example}

    \def\C{\mathbb{C}}
    \def\a{\alpha}
    \def\R{\mathbb{R}}
    
    \def\tilde{\widetilde}
    
    \def\o{\circ}
    
    \def\o{\circ}
    \def\0{\emptyset}
    \def\R{\mathbb{R}}

\newcommand{\CC}{{\bf C}}
\newcommand{\RR}{{\bf R}}
\newcommand{\liek}{{\bf k}}
\usepackage{tikz}
\usepackage{tikz-cd}
\usepackage{centernot}
\usepackage{mathtools}
\usepackage{ stmaryrd }
\usepackage{ mathrsfs }
\usepackage{fullpage}
\usepackage{hyperref}
\usepackage{cancel}
\usepackage{leftidx}
\usepackage{ upgreek }
\usetikzlibrary{matrix,arrows}
\hypersetup{colorlinks=true, linktoc=all, linkcolor=blue}

    \theoremstyle{definition}

    \def\C{\mathbb{C}}
    \def\a{\alpha}
    \def\R{\mathbb{R}}
    
    \def\tilde{\widetilde}
    
    \def\o{\circ}
    
    \def\o{\circ}
    \def\0{\emptyset}
    \def\R{\mathbb{R}}

\newcommand{\liet}{{\bf t}}

\usepackage[balance,small,nohug]{diagrams}
\usepackage[utf8]{inputenc}

\usepackage{amsthm}
\usepackage{amsmath}
\usepackage{amsfonts}
\usepackage{amssymb}
\usepackage{mathdots}
\usepackage{stmaryrd}

\usepackage{graphicx}
\usepackage{tikz-cd}
\usepackage{stmaryrd} 

\usepackage{theoremref}
\title{IMPLODED CROSS-SECTIONS}
\author{LISA JEFFREY AND SINA ZABANFAHM}
\begin{document}
\pdfoutput=1
    \begin{abstract} In this survey article,
we describe imploded cross-sections, which were developed in order
to solve the problem that the cross-section of a Hamiltonian 
$K$-space is usually not symplectic. 
In some specific examples we contrast the intersection 
homology of some imploded cross-sections with their 
homology intersection spaces.
      Moreover, we compute the homology of intersection spaces associated to the open cone of  a simply connected, smooth, oriented manifold and 
the suspension of such a manifold.
    \end{abstract}
\maketitle
\tableofcontents
\section{Introduction}

Let $(M, \omega)$ be a Hamiltonian $K$-manifold, where $K$ is 
a compact Lie group with maximal torus $T$. In other words,
we assume $M$ is equipped with a symplectic structure and with 
the action of a group $K$ which preserves the symplectic 
structure and is generated by the 
Hamiltonian flow of a collection of functions (the 
moment maps for the group action). Denote the 
Lie algebras of $K$ and   $T$  by $\liek$ and $\liet$ respectively.
Let the moment map be denoted $\Phi: M \to \liek^*$.

The preimage $\Phi^{-1}(\liet^*)$ is 
in general not a manifold, and even where it is a manifold
it is usually not symplectic.
The purpose of the imploded cross-section construction \cite{G-J-S}
 was  to construct 
a Hamiltonian $T$ space $M_{\rm impl}$ whose symplectic quotients 
(with respect to the $T$ action) correspond to the symplectic quotients of $M$ 
with respect to its $K$ action.

In this survey article,  we outline the construction of imploded cross-sections,
focusing on the universal imploded cross-section, which is the imploded 
cross section of the cotangent bundle of $K$.
The universal imploded cross-section could be seen as an analogue of the symplectic realization of $\liek^*$. 
%Ordinarily symplectic realizations are only considered for 
%Poisson manifolds of constant rank, and of course $\liek^*$ is not of constant% rank. 
 The universal imploded cross-section plays as important a role in the context of symplectic
manifolds as $\liek^*$ plays among Poisson manifolds.

   The imploded cross-section  can be studied
using intersection cohomology. We can also study it using the homology intersection 
space construction  \cite{Banagl-Hunsicker}. We examine
 differences between 
these constructions, focusing on the universal implosion.

The layout of this article is as follow.
In Section  \ref{s:implcr} we give general definitions and motivation 
for this construction.
%In Section \ref{s:ih} we recall the definition of intersection homology.
In Section \ref{s-one} we describe the uses of the 
imploded cross-section.
In Section \ref{s-two} we characterize the  universal
imploded cross-section.
In Section \ref{s-four} we describe the work of Dancer, Kirwan and Swann (see for 
example \cite{DKS:comp})
on hyperk\"ahler implosions.  We discuss the work of Safronov \cite{Safronov} relating imploded
cross sections with derived geometry.
We also describe   the work of Martens-Thaddeus \cite{MT} on the universal nonabelian symplectic cut. Next  we review the work of 
Howard, Manon and Millson \cite{HMM} about bending flows, and show 
how this work can be described in terms of imploded cross-sections.
 We outline
the work of Hilgert-Manon-Martens \cite{HMM} 
on 
symplectic contraction, and related work of  Jeremy Lane \cite{Lane-SC}.
Finally we review recent work of Hoffman and Lane on 
toric degenerations and integrable systems \cite{HL} which was
described in \cite{htalk} and \cite{ltalk}.

In Section \ref{symplectic-Implosion} we review
intersection homology and describe the intersection homology of a universal imploded
cross-section of $SU(3)$ (joint work of the first author with Nan-Kuo Ho \cite{HJ}).
In subsection \ref{conifold-transition-blow-up} we contrast intersection homology (IH) with  the homology intersection spaces (HI) of Banagl and
 Hunsicker \cite{Banagl-Hunsicker}.
In subsection \ref{ss:ihperv} we outline intersection homology. 
In subsection \ref{hisp} we outline the construction of homology
intersection spaces.
In subsection \ref{s3univ} we outline the construction of the 
universal imploded cross-section of $SU(3)$, while in 
subsection \ref{ihcone} we describe the intersection homology of a cone.
In subsection  \ref{theorem 1.1} we give a computation of the HI of the universal
imploded cross-section studied above. 
%from the IH results described  earlier in Section \ref{symplectic-Implosion}.
Finally, in the Appendix  \ref{s:suspension} we compute the  HI of suspensions,
by methods similar to those from subsection \ref{theorem 1.1}.

The authors thank Markus Banagl for guidance on Section 
\ref{symplectic-Implosion} of the paper,
and thank Reyer Sjamaar for suggestions regarding Section  \ref{s-four}.
They would also like to thank Ben Hoffman and Jeremy Lane  for providing the detailed description of their work that appears in Section \ref{hl} below.

\section{Imploded cross-sections} \label{s:implcr}
In this section and the next, we follow \cite{G-J-S}.

We fix $K$ to be a compact 
connected Lie group. Let  $(M,\omega)$ be a Hamiltonian $K$-manifold with equivariant moment map $\Upphi:M\rightarrow \liek$. Here $\liek$ denotes the Lie algebra of 
the group $K$. Moreover, we assume that $T$ is a maximal torus of $K$ and $\liet_+^*$ 
 is is a chosen  fundamental Weyl chamber in $\liet^*$.
 Here $\liet$ denotes the 
 Lie    algebra of the maximal torus $T$ of $K$.

Symplectic manifolds with Hamiltonian $K$ actions are parametrizations
of systems equipped with a symmetry group $K$. It is desirable to 
divide out by the symmetry group to obtain a simpler system.
However, in general the new system will not be symplectic. 

If $0$ is a regular value of the moment map, and 
$K$ acts freely on $\Phi^{-1}(0)$, then the  symplectic quotient 
$\Phi^{-1}(0)/K$   is a smooth
manifold. If $0$ is a regular value of the moment map,
then $K$ acts with finite stabilizers at 
all points of $\Phi^{-1}(0)$ and $\Phi^{-1}(0)/K$ is
an orbifold, in other words a topological space that is 
locally homeomorphic to the quotient space of an open subset of a smooth
manifold by a finite group action.

If we instead want a space that parametrizes not the symplectic
quotient at $0$, but rather  symplectic quotients at other orbits in $\liek^*$,
then it makes sense to look at $\Phi^{-1} (\mathcal{O}_\lambda)/K$
for a general orbit $\mathcal{O}_\lambda$ of the 
coadjoint action in $\liek$.
If we take the preimage of one such 
orbit of the coadjoint action under the moment map and then take the quotient by the 
action of $K$, in general we recover an orbifold, or in good
situations 
a smooth manifold.
For example, for $K = SO(3)$, the Lie algebra $\liek^*$ is
identified with $\RR^3$ and the coadjoint action of $K$   on it is
the rotation action. In this case, the orbits $\mathcal{O}_\lambda$ are
2-spheres with center $0$ 
through the points $ 0 \ne \lambda \in \liet^*$. 

\subsection{Symplectic cross-section theorem}

In general there is an open subset 
$U$ of  $\liet^*$ for which $\Phi^{-1}(U)/T$ is foliated by 
the symplectic quotients $\Phi^{-1}(\lambda)/T$ for 
$\lambda \in U$. Moreover,
for each $\lambda \in \liet^*$, we have
$\Phi^{-1}(\lambda)/T \cong \Phi^{-1} (\mathcal{O_\lambda})/K. $
The symplectic cross-section theorem of Guillemin and 
Sternberg (\cite{GS}, Section 41) states the following.

Let $M$ be a Hamiltonian $K$-manifold, where $K$ is a compact connected 
Lie group. Let $T$ be a maximal torus of $K$ with Lie algebra $\liet$.
Denote the moment map for the $K$ action on $M$ by $\Phi: M \to \liek^*$. 

Let $\alpha$ be a point in $\liet^*$. 
Let $p$ be a point in $M$  with $\alpha = \Phi(p)$. 
 Let $B_\epsilon (\alpha) $ be a ball of
radius $\epsilon$  in $ \liet^* $ around   $\alpha$.
By \cite{GS} Theorem 26.7, 
there is a $T$-invariant neighbourhood $U$ of $\Phi^{-1}(\alpha)$  in $M$
such that 
$$W= \Phi^{-1} \Bigl (B_\epsilon(\alpha)\Bigr ) \cap U $$
is a symplectic submanifold of $M$, where $B_\epsilon(\alpha)$ is 
a ball in $\liek^*$ with center $\alpha$ and radius $\epsilon$.
Then $W$ is $T$-invariant and 
the action of $T$ on $W$ is Hamiltonian and the moment map is 
just the restriction of $\Phi: W \to \liet^*$.
The space $W$ is called a slice for the $G$ action at $p$. 

Every $K$-orbit in $M$ 
that intersects $W$ intersects it in a single $T$-orbit
(\cite{GS}, Proposition 41.2) 

Finally $W$ can be reconstructed
from $\alpha$ and the isotropy representation of $T$ on 
$TW_p$, as a Hamiltonian $T$ space
(\cite{GS}, Theorem 41.2).

%QQQ
\subsection{Examples}

\begin{example}

\begin{enumerate}

\item
The space $ \CC P^1$, equivalently $S^2$,  is a  coadjoint orbit of 
the rotation group $SO(3).$
The moment map is the inclusion map into $\RR^3$. 
We may choose a maximal torus of $SO(3)$ so that the dual 
of the Lie algebra of 
the maximal torus is identified with the vertical axis $\{ (0,0,s)|s \in 
\RR\} $ in the dual of Lie algebra of $SO(3)$, which is identified with $\RR^3$. 
The preimage of the Lie algebra of $\liet^*$ under
$\Phi$ consists of two points, the north
and south poles. 
\item %The space $\CC P^n$ is a coadjoint orbit of $U(n+1)$.
The group $U(n+1)$ acts on the space $\CC P^n $ by right multiplication.
The moment map for this action is

$$\phi([z_0: \dots: z_n])_{ij}  =\frac{\sqrt{-1} z_i \bar{z_j} }{2 \pi 
\sum_{k=0}^n |z_k|^2}. $$
%Eliashberg-Traynor
The preimage of $\liet^*$ is the 
subspace where 
the image of the  moment map is a diagonal matrix 
with entries in $\sqrt{-1} \RR$.

This means $z_i \bar{z_j}  = 0 $ if $i \ne j$, and
hence that  at least one of $z_i $ or $z_j$ is $0$.
This is only possible if all but one of $z_i$ are $0$. 
These points of $\CC P^n$ are the fixed points of the action of the 
maximal torus by right multiplication. In 
this case the preimage of $\liet^*$ under
the moment map consists of isolated points.

\item We now consider the product of a collection of spheres.
Let $K = SU(2)$ act diagonally on a space $M$ which is the   product of 
$N$ copies of the 
sphere $\CC P^1$. 
The moment map is the 
sum 
$$ \Phi: (u_1, \dots, u_N) \mapsto \sum_{i=1}^N u_j. $$
Here, each  $u_j$ is regarded as a point of $\liek^*  \cong \RR^3$
(so the sum $\sum_j u_j$  makes sense).
Each $u_j$ satisfies $|u_j| = 1$ for each $j$ since each $u_j$ is a member of $S^2$.

In this case the condition that the moment map takes values in 
$\liet^*$ is that 
the sum $\sum_j u_j$ is in $\liet^*$, in other words that this sum is on 
a chosen axis in $\RR^3$, for example the vertical axis.

\end{enumerate}

Notice that the space of products of spheres 
has a subspace  of $\Phi^{-1} (\liet^*)$ for which the restriction 
of the symplectic form from the product of spheres is
not everywhere symplectic. This restriction 
is of course closed, but it may be degenerate.
Fix a point $x \in \Phi^{-1} (\liet^*) \subset M$.

Taking a basis for the
tangent space to $\Phi^{-1}(\liek^*)$ at
$x$,  denote by $\iota: \Phi^{-1}(\liet^*) \to 
\Phi^{-1} (\liek^*)$
the inclusion map, and its adjoint $\pi := \iota^*$ is the 
projection from the tangent space  to $\Phi^{-1} (\liek^*)$ at $x$ to 
the tangent space to $\Phi^{-1}(\liet^*)$ at $x$
 using the chosen invariant inner product.
Denote by $A_x$ the matrix that represents the 
symplectic form in this  basis. Then the condition that
the restriction of the symplectic form is degenerate at $x$
is that the matrix  $\pi \circ A_x  \circ \iota$ is degenerate,
in other words that the determinant of $\pi \circ A_x \circ \iota$ 
is $0$.
This means that one particular minor of this matrix is $0$. 
This minor is the determinant of a square submatrix of  real codimension $\dim (K) - \dim (T)$ $:= \ell$.
In appropriate coordinates on the 
tangent space to $M$ at $x$,  it is  the determinant of the first $2N -\ell$ rows and the first $2N-\ell$ columns of the matrix.

\end{example}

\begin{example} [Actions on coadjoint orbits] \cite{Lane-NF}

Jeremy Lane has studied the action of $SU(n-1) $ on a coadjoint orbit of $SU(n)$ where the action arisis from the inclusion of $SU(n-1)$ in $SU(n)$ and
the coadjoint action  of $SU(n)$ on the orbit. He finds that there are 
loci where the symplectic form on the orbit becomes degenerate when
restricting to  the dual of the Lie algebra of the maximal torus 
of $SU(n-1)$.
For example, Lane 
 finds that for $n=3$,  the symplectic cross-section of such an orbit is 
a Lagrangian submanifold  of  dimension $3$.
 \end{example}

When taking a symplectic quotient at a value of the moment map which is not a regular
value, the symplectic quotient is a stratified
symplectic space \cite{Sjamaar-Lerman}. 
In other words, the symplectic quotient is not a smooth 
manifold, but decomposes into strata each of which is a smooth
manifold and has a symplectic structure. The preimage of the Lie algebra
of the maximal torus under the moment map (called the symplectic 
cross-section) is also a stratified space, but the 
strata are not necessarily symplectic. 
The imploded cross-section is designed to
repair this so that each stratum of the preimage of the maximal torus 
under the moment map 
has a symplectic structure.

Define a relation $\sim$ on $\Upphi^{-1}({\liet}_+^*)$ as follows:

\begin{definition} \label{eqrel}
Let $K$ act on $\liek^*$ by the coadjoint action.
Then 
$m_1\sim m_2$ if there exists 
$k\in[K_{\Upphi(m_1)},K_{\Upphi(m_2)}]$ such that $k\cdot m_1=m_2$,
where $k \cdot m$ denotes the image of the action of $k$ on $m$.
%Note that ${\liet}^*_+$ denotes
%the fundamental Weyl chamber.
\end{definition}
  It turns out that this defines an equivalence relation on $\Upphi^{-1}({\liet^*})$. Indeed, by equivariance of the moment map $\Upphi$, $m_1\sim m_2$ implies that $K_{\Upphi(m_1)}=K_{\Upphi(m_2)}$ and therefore this 
equivalence relation is transitive.
\begin{definition}
The symplectic implosion $M$, denoted by $M_{impl}$, is defined as 
$$ M_{impl}=M/\sim, $$
where $\sim$ is the above equivalence relation.
This space is equipped with the quotient space topology.
\end{definition}
One can lift  the left action of $K$ on itself
 to a Hamiltonian action on the cotangent bundle $T^*K$. The implosion of 
the cotangent bundle, $(T^*K)_{impl}$, is called the {\em universal 
imploded cross-section} of $K$. The following theorem explains why this space is called ``universal":
\begin{theorem} \label{univsymimpl} (\cite{G-J-S}, Theorem 4.9)
For any Hamiltonian $K$-manifold $M$, there exists an isomorphism
\[M_{impl}\cong (M\times (T^*K)_{impl})\sslash_0 K,\]
where $\sslash_0$ denotes the symplectic quotient and 
 the symplectic quotient is with respect to the diagonal action of $K$.
\end{theorem}

The above theorem tells us that the imploded cross-section of $M$ is the 
same as the symplectic quotient at $0$  of the product of $M$ with the imploded cross-section of the cotangent bundle of $K$ (in other 
words the universal imploded cross-section of $K$).

%\section{Intersection homology} \label{s:ih}

\section{Universal  imploded cross-section} \label{s-one}
\subsection{Introduction}
Let $K$ be a compact connected Lie group.
The symplectic quotient of 
 a symplectic manifold equipped with a Hamiltonian $K$ action is a stratified symplectic space 
(in other words each stratum is equipped with a symplectic structure 
\cite{Sjamaar-Lerman}), but the preimage of the dual of the Lie algebra of the
 maximal torus $T$ under the moment map is not
necessarily symplectic. The imploded cross-section of such a manifold 
has the property  that the intersection of each stratum with the preimage of the 
Lie algebra of the dual of the maximal torus under the moment map is symplectic.  

\subsection{Universal implosion} \label{s-two}
The universal imploded cross-section is the imploded 
cross-section of the cotangent bundle of a Lie group. It has
the property that the symplectic quotient at $0$  of the product of a 
Hamiltonian $K$-manifold $M$ with the universal imploded cross-section of $T^* K$
 is
the imploded cross-section of $M$. The  definition of the 
universal imploded cross-section was
stated  in \cite{G-J-S}. 
%The definition and the result about the 
%universal imploded cross-section  will be outlined below  in Section
%\ref{symplectic-Implosion}.
\subsubsection{Universal imploded cross-section of $SU(2)$}

The universal imploded cross-section of $K= SU(2)$ is 
a copy of $\CC^2$, as shown in Example 4.7 of \cite{G-J-S}.
One way to describe this correspondence is that 
the cotangent bundle of $K$ may be identified with $K \times \liek^*$,
with the moment map given by projection on $\liek^*$.
Hence the preimage of $\liet^*$ is $K \times \liet^*$.
The implosion is obtained by collapsing the fiber above the identity
element $e$  by the action of the commutator subgroup $[K,K]$. 
But $[K,K]  = K$ so this action collapses $K \times \{0\} $ to 
 a point. 
Hence the imploded cross-section is $K \times \RR/ \sim$. Here
the equivalence relation identifies 
$(x,0) \sim (y,0)$. 
%and $(x,t) \sim (y,t) $ if 
%$t  = 0 $ 
If $s \ne 0 $ and $t \ne 0 $, then 
$(x,t) \sim (y,s) $ implies
$t = s$ and $x= y$.
This identifies the imploded cross-section as the cone on $S^3$, in other words $\CC^2$.

\subsubsection{Universal imploded cross-section of $SU(3)$} \label{s3univ1}

As described in Example 6.16 in \cite{G-J-S}, the universal imploded cross-section of $SU(3)$ has a structure of an irreducible affine complex variety which is given by
\[\{(z,w)\in \C^3\times \C^3 \mid  z\cdot w=0\}.\]
This space has an isolated singularity at $(0,0)$. It turns out that this space is homeomorphic to the open cone over the compact connected Riemannian manifold
\[Y=\{(z,w)\in \C^3\times \C^3\mid z \cdot w=0, \vert z\vert ^2+\vert w\vert ^2=1\}.\]

We show in \cite{HJ} (Theorem 4.2) that 
the space $Y$ decomposes as the union of two 
spaces $W$ and $X$, where 
 the disjoint union of two copies of $S^5$
is a deformation retraction of $W$, while 
$SU(3)$ is a deformation retraction of $X$ and 
the disjoint union of two copies of $SU(3)$ is a deformation retraction 
of 
$W  \cap X$.
 A Mayer-Vietoris sequence enables us
to compute the homology of $Y$, which completes
the computation of the intersection homology of $X$.

\subsubsection{Quasi-Hamiltonian analogue for $SU(2)$}

Quasi-Hamiltonian $K$-spaces were introduced
by Alekseev, Malkin and Meinrenken \cite{AMM}.  A
quasi-Hamiltonian $K$-space is a manifold equipped with 
a $K$-action and equipped with a 2-form $\omega$.
The form $\omega$ is neither closed nor nondegenerate,
but satisfies three key properties with respect to the $K$ action
(\cite{AMM}, Definition 2.2) which
are analogous to  certain properties of a 
Hamiltonian $K$-manifold  
(namely that it is equipped with a 2-form which is closed and nondegenerate,
and there is an equivariant  moment map $\mu: M \to \liek^*$). 

The space $K \times K$ is the quasi-Hamiltonian analogue of the cotangent
bundle $T^* K$, and has a group valued moment map (the commutator map
$\Phi: K \times K \to K$).
It is then possible to define an imploded cross-section for quasi-Hamiltonian
$K$-spaces in a manner analogous to the definition of imploded 
cross-sections for Hamiltonian $K$-spaces.

We now specialize to $K=SU(2)$ for the remainder of this section.
Under this definition,
the imploded cross-section of the quasi-Hamiltonian 
space  $K \times K$ (the double for the group $K$) is
isomorphic to $S^4$
(Proposition 2.29 of \cite{HJ1}). 

This is true because the strata of $K $ 
are $\sigma^0$ (the
stratum consisting of all elements of $K$ whose stabilizer is conjugate to $T$
under the adjoint action)
and $\sigma^\pm = \pm I $ (the 
stratum consisting of the elements of the center of $K$). We have
$$\Phi^{-1} (\sigma^0) \cong K \times (0,\pi) .$$
Of course $\Phi^{-1} (\sigma^\pm) $ is a copy of $K$ 
and $\Phi^{-1}(\sigma^\pm)/K$ is a point. 
These fit together to form the suspension of $K = SU(2) = S^3$, which 
is $S^4$. So the  imploded cross-section of $K \times K$ is 
$S^4$.

\section{Further work on symplectic implosion} \label{s-four}

In this section we describe some other examples of work on 
symplectic implosion. 

\subsection{Hyperk\"ahler implosion: the work of Dancer, Kirwan and Swann}
  Dancer, Kirwan, Swann and their collaborators defined hyperk\"ahler
analogues of symplectic implosion. In 
\cite{DKS:comp} these authors define a hyperk\"ahler reduction of the universal example
for $SU(n)$  in terms of quiver varieties.
 For example, the hyperk\"ahler reduction of the universal
example  is also a quiver variety. For the 
general definition of a quiver variety, see for example 
King \cite{King} or Ginzburg \cite{Ginzburg}.

In \cite{DKS:arbeit} these authors extend this treatment to the orthogonal and
symplectic groups. They
 discuss some of the ways in which  these cases
are different  from 
 $SU(n)$.
In \cite{DKS:ram} 
the authors show that the universal hyperk\"ahler
 imploded cross-section contains a hypertoric variety
(in other words a submanifold with a hyperk\"ahler structure
which is preserved by a torus action).
The last section of this paper outlines an alternative description involving gauge
theory and Nahm's equations. 

In \cite{DK:mult}, the hyperk\"ahler quotient of a space with a group-valued moment map is defined.
In \cite{DKS:nahm}, these authors relate hyperk\"ahler
implosion  to Nahm's equations. In \cite{DKS:twist} the authors study the  twistor
space (see for example \cite{MW}) associated to the hyperk\"ahler implosion of a Hamiltonian $K$-space, where
$K=SU(n)$. 
%The paper then discusses the possibility of extending this
%construction to other compact groups.

% Some of this work involves   
%Kirwan worked on symplectic implosion for nonreductive groups
%\cite{Kirwan}.

\subsection{Derived geometry and implosion: the work of Safronov}

Let $G$ be a reductive Lie group.
In \cite{Safronov}, Safronov shows that the universal implosion is equivalent to the 
hyperk\"ahler implosion of a stacky quotient.
The idea of shifted symplectic geometry was developed for a stack
by Pantev {\em et al.} \cite{Pantev}. Safronov interprets symplectic implosion in this context. Symplectic
implosion replaces a Hamiltonian $G$-space by a Hamiltonian 
$H$-space (where $H$ is the maximal torus of $G$) so that the symplectic quotients
at all level sets of the moment map are the same.   Group-valued 
implosions are defined.

Safronov gives a characterization of the universal symplectic implosion in terms of stacks
(using a compact Lie group $K$ and its complexification $G$, with 
Borel subgroup $B$ with Lie algebra $b$ and nilpotent subgroup $N$ with 
Lie algebra $n$, so that $n = [b,b]$).
Safronov obtains that the universal symplectic implosion of $G$ is  the stack 
$G \times_N b$. 

Safronov also provides an adjoint map  for implosion, which maps $H$-spaces to $G$-spaces.
This could be thought of as a one-sided inverse map.
%$j \circ i $ is the identity map (although 
%$i \circ j$ is not). 

Safronov shows in Theorem 3.11 \cite{Safronov} that quasi-Hamiltonian reductions of imploded cross-sections
with respect to a group $K$ are the same as covers of quasi-Hamiltonian reductions with respect to the maximal torus $H$ 
along a coadjoint orbit.

\subsection{Symplectic cuts: the work of Martens and Thaddeus}
The work of  Martens and Thaddeus \cite{MT} defines a ``universal
nonabelian symplectic cut", the 
``nonabelian symplectic cut"  of the cotangent bundle of a compact Lie group $K$.
The nonabelian symplectic cut of a Hamiltonian $K$-manifold $M$ is then 
defined as the symplectic quotient of the product of $M$ and 
the universal nonabelian symplectic cut. The universal nonabelian 
symplectic cut has a symplectic description as the 
symplectic cut of $T^*K$ according to a  polytope $P$  (see equation (12)
of \cite{MT} for the 
definition). There is also an algebraic-geometric characterization of the universal
nonabelian symplectic cut
(see equation (14) of \cite{MT}).

\subsection{The work of Howard-Manon-Millson} \label{ss:bendingflows}

In \cite{HoMM},
 these authors consider the space of polygonal linkages 
in $\RR^3$, in other words
$m$-sided polygons in $\RR^3$ with fixed
side lengths.
In this paper, the authors restrict to the group $K=SU(2)$
whose Lie algebra is $\RR^3$.
 This space can be identified with the 
symplectic quotient of 
the Grassmannian of two-planes in $\CC^m$ by 
the action of $T^m$, where $T$ is the maximal torus of $K$.

  This space is equipped with a torus action
(``bending flows''). The torus action  fixes
one part  of a polygon  (the part on one side of a 
diagonal) and rotates the rest around
that  diagonal.  The difficulty is that
there is a set of measure zero where  the 
torus action  fails to be defined, the
set where some  diagonal 
has length $0$. 

%This is also called the space of imploded spin-framed $n$-gons.

%Identify where bending flow not defined

%Kamiyama-Yoshida modify so  bending flow  is defeind everywhere

The space of $m$-gon linkages was originally 
described in \cite{KM} and was given a symplectic structure. 
In the paper \cite{HoMM}, the authors
identify the space of  polygonal linkages  in $\RR^3$ with the 
symplectic quotient of a  Grassmannian by a torus
action (as described above). 

They show in Section 3 of their paper that 
this topological space may be given an alternative
description as an imploded cross-section.
One may use the homeomorphism between the Grassmannian of 
2-planes in $\CC^m$ and the symplectic quotient
of the imploded cross-section of cotangent bundle of $K^m$ by the
left diagonal action of $K$.

\subsection{Symplectic contractions: the work of Hilgert-Manon-Martens and Lane}

Hilgert, Manon and Martens  \cite{HMM} define  a symplectic contraction from one
 Hamiltonian space to another, which is a continuous
surjective map to a new Hamiltonian space whose restriction to an open dense 
subset is   a symplectomorphism. They give an explicit formula for the symplectic 
contraction map. They are able to show that the symplectic
contraction map $\Phi$ maps  a dense subset 
of $M$ symplectomorphically onto a dense subset of the image of $\Phi$.    They interpret the 
Gelfand-Zeitlin system on a coadjoint orbit in this language. Their construction uses symplectic reduction and symplectic
implosion.

Jeremy Lane \cite{Lane-SC} gives an 
alternative definition of  a symplectic contraction $M^{sc}$ of a symplectic 
manifold $M$. He shows that 
his definition is equivalent to the definition 
given by Hilgert, Manon and Martens.
He exhibits   the symplectic contraction map $\Phi: M \to M^{sc}$ 
as a surjective Poisson map, a property that is not immediately
obvious from the definition given in  \cite{HMM}.
Lane  also shows that  the symplectic contraction  is equipped 
with  a 
Poisson algebra of smooth functions. 
He identifies the symplectic contraction with the 
quotient space obtained by subdividing $M$ into suitable coisotropic
submanifolds and quotienting each
coisotropic submanifold by the null foliation of the 
restriction of the symplectic form  to it.

Lane proves that the symplectic contraction map sends $M$ to $M^{sc}$ in a way that is a homeomorphism which preserves strata.  This enables him
to define a smooth structure on the symplectic contraction (the 
quotient space), so he is able to define a smooth structure
on a singular space. 

The Poisson structure on the symplectic 
contraction pulls back from the Poisson structure on the original 
manifold. 
Lane uses symplectic contractions to study Gelfand-Zeitlin systems:
these are symplectic manifolds
equipped with a multiplicity-free Hamiltonian 
$U(n)$ action.  Gelfand-Zeitlin systems have 
many similarities with toric manifolds 
but are not always equipped with the Hamiltonian 
 action of a torus of half the dimension of the manifold.
%He proves that there are some associated
% toric manifolds (``symplectic pieces'').

\subsection{Canonical bases and collective integrable systems: the work of  Hoffman and Lane} \cite{HL,htalk,ltalk} \label{hl}

Let $K$ be a compact connected Lie group and let $(M,\omega,\mu)$ be a Hamiltonian $K$-manifold. A classical 
 problem in symplectic geometry asks: Beginning with the action of $K$,
 can one construct a Hamiltonian action of the compact torus $\mathbb{T} = (S^1)^m\times T$ on $M$  (with as small a kernel as possible)?
Here $m$ is half the dimension of a regular coadjoint orbit of $K$ and $T$ is the maximal torus of $K$. 

This question was first answered by Guillemin and Sternberg 
\cite{GSGC} in the case when $K$ is a unitary group $U(n)$ or an orthogonal group $O(n)$. Their solution involved the construction of the now-famous Gelfand-Zeitlin integrable system on 
$\liek^*$; they show that there is 
a Hamiltonian $\mathbb{T}$ action on $M$ whose moment map is the composition of $ M\xrightarrow{\mu}  \liek^*$ with the moment map $\liek^* \xrightarrow{\Psi} 
\operatorname{Lie}(\mathbb{T})^*$ for the Gelfand-Zeitlin system.
A second approach, involving toric degenerations, was used by Harada-Kaveh \cite{HK} in the case that $M$ is a smooth projective variety and $\omega$ is the Fubini-Study form.

 The following result of Hoffman and Lane shows how to solve this problem in the general case.% In some sense it unifies the approaches of Guillemin-Sternberg and Harada-Kaveh.

\begin{theorem}\cite{HL}
    Let $K$ be a compact connected Lie group. There exists a continuous map 
$\Psi \colon \liek^*\to \operatorname{Lie}(\mathbb{T})^*$ so that, for any Hamiltonian $K$-manifold $(M,\omega,\mu),$ there is a commuting diagram
    \begin{equation}\label{eqn; thm 1 diagram}
	    \begin{tikzcd}
           M \ar[d,"\mu"]\ar[r,"\phi"] & X \ar[d] \\
           \liek^*\ar[r,"\Psi"] & \operatorname{Lie}(\mathbb{T})^*.
        \end{tikzcd}
    \end{equation}
    Here,
    \begin{itemize}
        \item $X$ is a singular Hamiltonian $\mathbb{T}$-space and the vertical arrow on the right is its moment map.
        \item $\phi$  is a continuous, proper, $T$-equivariant, surjective map that is a symplectomorphism from a dense subset of $M$ onto its image.
    \end{itemize}
    Moreover:
    \begin{itemize}
        \item The map $\Psi\circ \mu$ generates a Hamiltonian $\mathbb{T}$-action on a dense subset of $M$. 
        \item If $M$ is multiplicity-free, then the action of $\mathbb{T}$ on a dense subset of $M$ is completely integrable.
        \item If $M$ is compact and connected, then its image in $\operatorname{Lie}(\mathbb{T})^*$ is a convex polytope.
    \end{itemize}
\end{theorem}

Although we do not go into details of their construction here, 
the proof of this result uses symplectic implosion in a fundamental way.  
The key ingredient is the identification of the universal
 symplectic implosion of $T^*K$ with the geometric invariant theory quotient $G\sslash N$, where 
$G$ is the complexification of $K$ and $N$ is a maximal unipotent subgroup 
of $G$
\cite{G-J-S}.  It is known that the affine variety $G\sslash N$ admits 
toric degenerations.  Hoffman and Lane show that it is 
possible to integrate gradient-Hamiltonian flows of these 
degenerations and thus obtain completely integrable Hamiltonian torus actions on the universal symplectic implosion of $T^*K$. In contrast with previous works~\cite{NNU, HK} on gradient-Hamiltonian flows and toric degenerations, the variety $G\sslash N$ 
is neither smooth nor projective. In order to prove their result, Hoffman and Lane develop new techniques for controlling gradient-Hamiltonian vector fields in this setting. 

The degenerations of $G\sslash N$ used in the construction are 
degenerations to specific affine toric varieties $X_C$, as in~\cite{C}. 
Each $X_C$ is associated with the semigroup of integral points in a convex rational polyhedral cone 
$C\subset\operatorname{Lie}(\mathbb{T})^*$ 
described by Berenstein-Zelevinsky and Littelmann \cite{BZ, L}. 
The lattice points of $C$ are in bijection with 
the elements of the Kashiwara-Lusztig 
dual canonical basis~\cite{K, Lu} of $\mathbb{C}[G\sslash N]$. Under the map $\Psi$ constructed by Hoffman and Lane, the image $\Psi(\liek^*)$ is equal to $C$.
The cone $C$ comes with a natural linear projection to the positive Weyl chamber $\liet^*_+$. This projection fits into a commuting diagram:
\begin{equation}
	    \begin{tikzcd}
           \liek^*\ar[r,"\Psi"]\ar[d,"/K"] & \operatorname{Lie}(\mathbb{T})^*\ar[d]\\
           \liet_+ \ar[r,"="] & \liet_+.
        \end{tikzcd}
\end{equation}
Combining this diagram with the one in the theorem 
above, one readily observes that when $M$ is compact and connected, 
the image $\Psi\circ\mu(M)$ is precisely the intersection of $C$ with the pre-image of the Kirwan polytope of $M$ under the projection $\operatorname{Lie}(\mathbb{T})^* \to \liet^*_+$.

\section{Topological invariants of the universal imploded cross-section}\label{symplectic-Implosion}

%In Section \ref{theorem 1.1}, we give an alternative p.

% \section{Intersection homology}
In this section we describe some  invariants
of imploded cross-sections. We specialize to the universal 
imploded cross-section, because general symplectic quotients of 
imploded cross-sections can be obtained as symplectic quotients
of the direct product of an arbitrary Hamiltonian $K$-manifold
and the universal imploded cross-section of $K$. 
Because the examples are easier to compute, we specialize
to the universal imploded cross-section of $K=SU(3)$. 
We compare intersection homology ($IH$) with homology 
intersection spaces ($HI$). 

This section is subdivided as follows. 
In subsection \ref{ss:conifold} we describe the conifold transition.
In subsection \ref{ss:ihperv} we survey intersection 
homology and perversities.
In subsection  \ref{hisp} we summarize homology intersection spaces.
In subsection \ref{s3univ} we compute the intersection 
homology of the universal example for $SU(3)$. 
In subsection \ref{ihcone} we summarize the intersection 
homology of a cone.

\subsection{Conifold transition and blow-up manifold}\label{conifold-transition-blow-up}  \label{ss:conifold}

%\label{s-six}

In this section we define the conifold transition and the 
blow-up spaces, following Banagl \cite{Banagl}.

 These spaces  provide
 an approach to studying Poincar\'e duality on singular spaces. 
An object called a perversity will be defined below in Definition (\ref{d:perv}).
Given a perversity $\overline{p}$,
one  may associate a CW complex $I^{\overline{p}}X$ to a 
certain class of singular spaces $X$ . For our purposes, we only need to understand this construction in the case that $X$ 
is a Thom-Mather pseudomanifold of depth 1 with a trivial link bundle. The definition of Thom-Mather stratified spaces is given with more generality in \cite{Albin}. The following definition appeared in \cite{Banagl-Chriestenson}:
    \begin{definition}
    A depth one pseudomanifold $X$ with singularity $\Sigma$ is a pair 
$(X,\Sigma)$, where \begin{enumerate}
        \item $\Sigma$ is understood to be a closed subspace and a smooth manifold of codimension at least 2.
        \item $X\setminus \Sigma$ is a smooth manifold which is dense in $X$.
        \item $\Sigma$ possesses control data consisting of a tube $T\subset X$ around $\Sigma$ which is an open set in $X$ together with two maps:
        \begin{align*}
            &\pi: T\longrightarrow \Sigma\\
            &\rho:T\longrightarrow [0,\infty)
        \end{align*}
        such that $\pi$ is a continuous retraction and $\rho$ is a continuous distance function such that $\rho^{-1}(0)=\Sigma$. Moreover, it is required that $(\pi,\rho):T\setminus \Sigma\longrightarrow \Sigma \times (0,\infty)$ is a smooth submersion.
    \end{enumerate}
    \end{definition}
    \begin{definition} \label{deflcl}
Let $L$ be a simply connected, smooth manifold.
Let  $c^\o(L)$ be the open cone on $L$.
\end{definition}
\begin{remark}
    Notice that, when $L$ is a smooth manifold, 
the cone $c^\o(L)$ on $L$  is a depth 1 Thom-Mather pseudomanifold, with $v$ (the vertex of the cone) as its singularity.
    The link bundle of a depth 1 pseudomanifold is defined in Proposition $8.2$ in 
\cite{Banagl-Chriestenson}. In the case that $X=c^\o(L)$, the link bundle is as follows: 
    \[L\rightarrow v. \]

    \end{remark}
    Carefully following \cite{Banagl-Hunsicker}, we write down the construction of the conifold transition and the blow-up manifold associated to 
$(X,\Sigma)$.

 Take a tubular neighborhood $N$ around the singularity $\Sigma$ and fix a diffeomorphism
\[\theta:N \setminus \Sigma\cong L\times \Sigma \times (0,1).\]
Define the blow-up manifold to be:
\[\overline{M}=(X \setminus \Sigma)\cup_\theta (L\times \Sigma\times[0,1)).\]
Define the conifold transition to be:
\begin{align}\label{conifold-transition}
    CT(X)=\frac{(X \setminus \Sigma)\cup_\theta (L\times\Sigma \times [0,1))}{(z,y,0)\sim (z,y^\prime,0)}
\end{align}
for all $z\in L$, and for all $y,y^\prime\in \Sigma$.
\begin{remark}\label{cone-structure}
Following this construction, one can see that when the singularity $\Sigma$ is a point, $CT(X)=\overline{M}$ is a manifold with boundary $L$. In particular, when $X$ is $c^\o(L)$ for some smooth manifold $L$, we have \[CT(X)=\overline{M}\cong L\times [0,1). \]
\end{remark}

\subsection{Intersection homology and perversities} \label{ss:ihperv}

In this section we outline intersection homology, which reduces to ordinary homology
for smooth manifolds. See the books \cite{friedman} (Chapter 2)  and \cite{Kirwan-Woolf}  for general 
background information on intersection homology.

Taking the quotient of a smooth manifold by a group action 
produces singularities unless the group acts freely.
The resulting topological spaces fail to satisfy 
standard topological properties  such as Poincar\'e duality.
Intersection homology and intersection cohomology were introduced
by Goresky and MacPherson \cite{Goresky-MacPherson} in an effort to 
preserve some of the topological properties that
are familiar in the setting of smooth manifolds.  Intersection homology
satisfies the ``K\"ahler package", which includes Poincar\'e duality,
the hard Lefschetz theorem and  the Lefschetz hyperplane theorem. This is 
the most we can hope to retain when we leave the  the setting of  smooth manifolds.

Since imploded cross-sections are singular, it is reasonable
to investigate their topological properties using 
tools like intersection homology which are adapted to singular
spaces, rather than to try to study the ordinary homology of 
these spaces.

\begin{definition} \label{d:perv}
A perversity is a map 
$$p: \{ 2,3, \dots, \} \to \{ 0, 1, \dots\} $$
satisfying
\begin{equation} \label{cond1}
p(2) = 0 
\end{equation}
and
\begin{equation} \label{cond2}
 p(k) \le p(k+1) \le p(k) + 1 \end{equation}
for all $k \ge 2.$ 
\end{definition}

Throughout this article  $\overline{p},\overline{q}$ are considered to be extended perversities, 
which are just  sequences of integers (see \cite{Banagl-Hunsicker}, Section 3). 

\begin{definition} 
The intersection homology of $X$ is the homology of the chain complex of $p$-allowable
chains of $X$, where $p$ is a perversity.
Here a  $j$-simplex $\sigma$ is $p$-allowable if 
\begin{equation} {\rm dim} (\sigma \cap X_{n-k}) \le j-k + p(k) 
\end{equation}
for all $k \ge 2$.
\end{definition}

\subsection{Homology intersection spaces} \label{hisp}

\vspace{0.5truein}

%\noindent{\em Notations and conventions:} 

 The definition of intersection spaces was given 
 in \cite{Banagl}.
In this subsection we will describe how to construct $I^{\overline{p}}X$, the perversity $\overline{p}$ intersection space associated to $X$ when $(X,\Sigma)$ is a depth 1 pseudomanifold 
where the link $L$ of the singular stratum $\Sigma$ is simply 
connected and the link bundle is 
the product bundle $L \times \Sigma \to \Sigma$ 
(as defined in the first section of  \cite{Banagl-Hunsicker}).
The definition of the intersection space is given below in equation 
(\ref{intspacedef}).

Let $l:=\dim L$ and set $k:=l-\overline{p}(l+1)$. \footnote{Note that by the dimension of a manifold we always mean the real dimension.}
\begin{definition} \label{llessdef}
Let  $L_{<k}$ be  the union of all strata of $L$ of degree less than $k$.
\end{definition}
Assume $f:L_{<k}\longrightarrow L$ is a stage $k$ Moore approximation of $L$ 
(see \cite{Banagl-Hunsicker}, Definition 3.1 for the definition).
 In other words,
the homology groups $H_i(L_{<k})$ are $0$ for $i\geq k$ and \[f_*:H_i(L_{<k})\longrightarrow H_i(L)\] is an isomorphism for $i<k$. Define the map $g:L_{<k}\times \Sigma\longrightarrow M$ to be the composition:
\begin{equation} \label{Moore-map}
    L_{<k}\times\Sigma\xrightarrow{f\times id_\Sigma}L\times \Sigma= \partial \overline{M}\hookrightarrow \overline{M}.
\end{equation}

%QQQQ
\begin{definition} \label{intspacedef}
In the notation introduced above, 
the perversity $\overline{p}$ intersection space $I^{\overline{p}}X$ is defined to be:
\begin{align} %\label{intspacedef}
    I^{\overline{p}} X ={\rm cone} (g)=\overline{M}\cup_g c(L_{<k}\times \Sigma).
\end{align}
\end{definition}

 The notation $I^{\overline{p}}X$ stands for the perversity $\overline{p}$ intersection space associated to $X$ (as introduced in \cite{Banagl}). See 
the definition given by equation  (\ref{intspacedef}) above. 

\begin{definition} (Homology intersection space) \cite{Banagl-Hunsicker}
The homology  $\tilde{H}I_i^{\overline{p}}(X)$ is defined by
\[\tilde{H}I_i^{\overline{p}}(X)=\tilde{H}_i(I^{\overline{p}}X;\R)\]
where by $\tilde{H}_*(X)$ we mean the reduced (singular) homology of $X$. 
\end{definition}

\begin{remark} \label{prevremark}
When $X$ is a stratified pseudomanifold of dimension $n$ with an isolated singularity, the following 
formulas are available for the homology of the 
intersection space $\tilde{H}I^p(X)$ (\cite{Banagl-Maxim}, p. 221):
\begin{align}\label{isolated-sing}
{H}I_*^p(X)=\begin{cases}
{H}_j(\overline M,\partial{\overline M}) & j<k\\
{H}_j(\overline M) & j>k
\end{cases}
\end{align}
where $k:=n-1-p(n)$ and $\overline{M}$ is the blow-up manifold associated to the space $X$ (as defined in subsection \ref{conifold-transition-blow-up}).
For the dimension $k$ homology, the following diagram with exact rows and columns exists:
\begin{equation*}
\begin{diagram}\
      &    &               &    &0		 &    &                               &    &   \\
      &    &               &    &\dTo		 &    &                               &    &   \\
    0 &\rTo& {\rm ker}\it{} \bigl(\tilde H_k(\overline M)\rightarrow \tilde H_k(\overline M,\partial \overline M)\bigr )&\rTo& \tilde{H}_k(\overline M) &\rTo&  I\tilde{H}_k(X) &\rTo& 0 \\
      &	   &	           &    &\dTo    	 &    & 			      &    &   \\
      &    &               &    & \tilde{H}I_k(X)&    &                               &    &   \\
      &    &               &    &\dTo		 &    &                               &    &   \\
      &    &               &    &{\rm im}\it{}\bigl(\tilde{H}_k(\overline M,\partial\overline M\bigr )\rightarrow H_{k-1}(\partial \overline M))&    &                               &    &   \\
      &    &               &    &\dTo		 &    &                               &    &   \\
      &    &               &    & 0		 &    &                               &    &
\end{diagram} 
\end{equation*} \hfill $\square$
\end{remark}
Remark \ref{prevremark}
 provides a proof of Theorem \ref{main-result}. 
When $X=c^o(L)$, the blow-up manifold $\overline{M}$ 
associated to $X$ is equal to $L\times[0,1)$ 
(as explained in Remark \ref{cone-structure}). In particular, notice that 
\[\tilde{H}_*(\overline M)=\tilde{H}_*(L),\] and \[\tilde H_*(\overline{M},\partial \overline{M})=\tilde{H}_*(\overline M/\partial\overline M)=\tilde{H}_*(c^o(L))=0.\]

%\subsection{Intersection homology of universal imploded cross-sections}

%\subsection{Universal imploded cross section for $SU(3)$} \label{s3univ}
%As described in Example 6.16 in \cite{G-J-S}, the universal imploded cross sec%tion of $SU(3)$ has a structure 
%of an irreducible affine complex variety which is given by
%\[\{(z,w)\in \C^3\times \C^3 \mid  z\cdot w=0\}.\]
%This space 
%has an isolated singularity at $(0,0)$. It turns out that this space is homeom%orphic to the open cone over the compact connected Riemannian manifold
%\[Y=\{(z,w)\in \C^3\times 
%\C^3\mid z \cdot w=0, \vert z\vert ^2+\vert w\vert ^%2=1\}.\]

\subsection{Universal imploded cross-section for $SU(3)$} \label{s3univ}

The universal imploded cross-section  $(T^*SU(3))_{\rm{impl}}$
of $SU(3)$ was described in Section \ref{s3univ1} above.
The middle perversity intersection homology of
the universal imploded cross-section $(T^*SU(3))_{\rm{impl}}$ is calculated in \cite{HJ} and it is given by:
\begin{align*}\label{homology-group}
I\tilde{H}^m_j((T^*SU(3))_{\rm{impl}})=\begin{cases}
\R & j=4\\
0 & \rm{otherwise}.
\end{cases}
\end{align*}
This was  done in \cite{HJ} by first computing the homology of $Y$ by a Mayer-Vietoris argument, and then applying at result to compute the interseection homology.

Comparing this with the result of Corollary \ref{SU(3)} below,
 we observe that the homology theories ${HI}^m$ and $IH^m$ do not agree on $(T^*SU(3))_{\rm{impl}}$.

\subsection{Intersection homology of a cone} \label{ihcone}

In this section we describe  
the intersection homology for a cone.

\begin{definition}
By $c^\o(X)$, the open cone over a topological space
 $X$, we mean the quotient space
\begin{equation} \label{czero}
c^\o(X)=\frac{(0,1]\times X}{(1,x) \sim (1,x^\prime)}.
\end{equation}
On the other hand, by $c(X)$ we mean the closed cone over $X$.
\end{definition}

Given any smooth manifold $L$ and perversity function $p$, the perversity $p$ intersection homology groups of $c^\o(L)$ are given by ~(\cite{Kirwan-Woolf}, p. 58):
\[I\tilde{H}^p_j(c^{\o}(L))=\begin{cases}
\tilde{H}_j(L) & j<l-p(l+1)\\
0 & \rm{otherwise}.
\end{cases}\]
By comparing this with the result of Theorem \ref{main-result} below,
 we see that the homology theories $\tilde{H}I^p$ and $I\tilde{H}^p$ usually do not agree on open cones over simply connected, smooth oriented manifolds.

First we are going to prove  two  general results.
\begin{definition} \label{intspacecone}
Let $X$ be as defined above.
In terms of the notation $L$ introduced above,
the  intersection space of $X$ is
$$
I^{\overline{p}}X=\frac{L\times [0,1) \sqcup c(L_{<k})}{(f(x),0)\sim (x,0)}  $$
where $f:L_{<k}\longrightarrow L$ is a stage $k=l-\overline{p}(l+1)$ Moore approximation of $L$.
Here the spaces  $L$ and $c_\o(L)$ were defined in Definition 
\ref{deflcl} above, and 
  the space $L_{<k}$ was defined in Definition \ref{llessdef}.
\end{definition}

\begin{theorem}\label{main-result}
Let $L$ be a simply connected  smooth manifold of dimension $l$, and let $$X=c^\o(L)$$ denote
the open cone on $L$, where
$L$ is as in Definition \ref{intspacecone}.
Assume that $\overline{p}$ is an (extended) perversity. Then
the homology intersection space associated to $X$ and $\overline{p}$ is
\begin{align*}\tilde{H}I^{\overline{p}}_j(X)=\begin{cases} 
      0 & 0 < j<l-\overline{p}(l+1) \\
      \tilde{H}_j(L) & \rm{otherwise}.
   \end{cases}\end{align*}
\end{theorem}

In the  following Lemma  we  compute the homology of the middle perversity intersection space associated to the universal imploded cross-section of $SU(3)$, 
denoted by $(T^*SU(3))_{\rm{impl}}$. 
\begin{lemma} \label{lem}
As shown in \cite{HJ}, the imploded cross-section
$(T^*SU(3))_{\rm{impl}}$ is homeomorphic to the open cone $c^\o(Y)$ where \begin{align}\label{Base-of-the-cone}
    Y=\{(z,w)\in \C^3\times \C^3\mid z\cdot w=0, \vert z\vert ^2+\vert w\vert ^2=1\}
\end{align}
is a compact Riemannian manifold of $\dim_\R(Y)=9$.
(Here $z \cdot w$ denotes the usual inner product on $\CC^3.$)  Moreover, the
reduced  homology groups of $Y$ are given by:
\begin{align}\label{homology-group}
\tilde{H}_j(Y)=\begin{cases}
\R & j=4,5,9\\
0 & \rm{otherwise}.
\end{cases}
\end{align}
\end{lemma}

Using the Lemma \ref{lem},
the corollary below
is thus a direct consequence of Theorem \ref{main-result}.\footnote{Throughout this paper, the letter $m$ denotes the lower middle perversity -- 
see for example \cite{Kirwan-Woolf}.}
\begin{corollary}\label{SU(3)}
The homology intersection space of the imploded cross-section of $SU(3)$ is 
\[\tilde{H}I^m_j((T^*SU(3))_{\rm{impl}})=\begin{cases} 
      \R  & j=5,9         \\
      0 & \rm{otherwise}.
   \end{cases}\]
\end{corollary}

%\subsection{Homology of intersection spaces}}

%\begin{remark}
%It would be worthwhile to compare   
% the intersection homology of  universal imploded cross-sections
%for groups other than $SU(3)$ (for example $SU(n)$ when $n \ge 4$)
%with the homology of the corresponding intersection spaces. It may
%be possible to establish that the two are not the same even if we are
%not able to explicitly
%construct  either of them.

%Another objective is  to see what consequences 
%can be deduced  about general imploded cross
%sections from  results about universal imploded cross sections,
%using Theorem  \ref{univsymimpl}.
%\end{remark}

\begin{rk}
Observe  that homology of intersection spaces is a homotopy invariant,
whereas intersection homology is not.
\end{rk}
\subsection{Proof of Theorem \ref{main-result}}\label{theorem 1.1}

In this section we give the proof of Theorem \ref{main-result}, which
gives the intersection homology of cones of smooth oriented manifolds
$L$.
\begin{proof}
Following  Remark \ref{cone-structure}, we have
\[I^{\overline{p}}X=L\times [0,1)\sqcup c(L_{<k})/\sim\]
where $k=l-\overline{p}(l+1)$ and the equivalence relation is given by:
\[ (x,0)\sim \bigl (f(x),0\bigr)\mbox{,\quad} \forall x\in L_{<k}.\]
Here $f:L_{<k}\longrightarrow L$ is a stage $k$ Moore approximation of $L$
(see for example \cite{Banagl-Chriestenson}). Define two open sets $A$ and $B$ as follows:
\begin{align}\label{set A}
A=\frac{L\times[0,1)\sqcup L_{<k}\times [0,\frac{1}{2}+\epsilon)}{(f(x),0)\sim (x,0)},
\end{align}
\begin{align}\label{set B}
B=C(L_{<k})\setminus \Bigl (L_{<k}\times [0,\frac{1}{2}-\epsilon) \Bigr ).
\end{align}
We observe that $B$ is contractible, as it is the preimage of $[1, 1/2 + \epsilon)$ in $c(L_{<k})$ 
under the cone map sending $(x,y) \in [0,1] \times L_{<k}$ to $[0,1]$.
Hence it is homeomorphic to the cone on $L_{<k}$, so it is contractible. The set $A$ is homotopy equivalent to $L$, because the identification 
map $f$ identifies each point in $L_{<k}$ to a point in $L$. Moreover, we observe that $A\cap B$ deformation retracts to $L_{<k}$, as $A\cap B$ is
homeomorphic to $L_{<k}\times(1/2-\epsilon,1/2 +\epsilon)$.\\
Writing the Mayer-Vietoris sequence, we have the following:

\textit{Case I}: $j-1 \geq k$

In this case, the Mayer-Vietoris sequence gives
$ 0 \to \tilde{H}_j(L) \to  \tilde{H}_j(I^{ \bar{p}} (X))  \to 0$
(because  $\tilde{H}_j(L_{<k}) = \tilde{H}_{j-1} (L_{<k}) = 0 $).
Hence $I^{\bar{p}} (X) \cong  \tilde{H}_j(L_{<k}) = 0 $. 

\textit{Case II}: $j=k$

The Mayer-Vietoris sequence gives

 \begin{equation*}
\begin{diagram}  \dots &\rTo& 0 = \tilde{H}_k(L_{<k}) &\rTo^{f_*}& \tilde{H}_k(L) &\rTo^{\alpha_k}&  \tilde{H}_k (I^{\bar{p}}(X)) &\rTo^{\beta_k}& \cr  
\dots &\rTo& \tilde{H}_{k-1} (L_{<k})  &\rTo^{f_*}& \tilde{H}_{k-1} (L) &\rTo^{\alpha_{k-1}}&
  \tilde{H}_{k-1} (I^{\bar{p}}(X) ) &\rTo^{\beta_{k-1}}& \cr
\dots &\rTo& \tilde{H}_{k-2} (L_{<k}) &\rTo^{f_*}&  \tilde{H}_{k-2}(L) &\rTo^{\alpha_{k-2}}&
\tilde{H}_{k-2} (I^{\bar{p}}(X) ) &\rTo^{\beta_{k-2}}&\dots \end{diagram} \end{equation*}
The maps $f_*: \tilde{H}_j(L_{ <k } ) \to \tilde{H}_j(L) $ are isomorphisms 
for $j \le k -1$. This implies   $\alpha_j$
$= \beta_j = 0 $ for $j \le k-1$. 
Also $\beta_k = 0 $,   so $\tilde{H}_k(I^{\bar{p}}(X) ) \cong \tilde{H}_k (L)$.

\textit{Case III}: $j<k$

Since $f_*:\tilde{H}_j(L_{<k})\rightarrow \tilde{H}_j(L\ )$ is an isomorphism for $j < k$, $\alpha_j = $ 
$\beta_j = 0 $ in this range. This implies the Mayer-Vietoris sequence gives
$$0 \to \tilde{H}_j(I^{\bar{p}}X)  \to 0 $$
which implies $\tilde{H}_j (I^{\bar{p}} (X) ) = 0 $ for $j \le k-1$. \
This completes the proof of the theorem. 
\end{proof}
\begin{remark}
When $X=c^\o(L)$, $I^{\overline{p}}X$ deformation retracts to ${\rm Cone}(f)$ where by ${\rm Cone}(f)$ we mean the mapping cone of 
$f:L_{<k}\longrightarrow L$. Therefore Theorem \ref{main-result} gives the homology groups of ${ \rm Cone}(f)$.
\end{remark}

%\begin{remark}
%Note that mirror symmetry interchanges intersection 
%homology of one class of spaces with  the homology of intersection 
%spaces of a different (mirror) class of spaces. 
%See for example \cite{Banagl}, Chapter 3.8 (particularly 
%Theorem 3.9 and Corollary 3.14).   From this point of view it is interesting
%to understand the difference between intersection 
%homology  and homology of intersection spaces for the imploded cross-section.
%\end{remark}

\appendix
\section{Intersection space associated to a suspension} \label{s:suspension}
The authors are not aware of 
examples of imploded cross-sections which 
are suspensions. 
The  material below  is included nonetheless because 
the  homology of intersection spaces of 
suspensions can be treated using the same techniques
as are used to characterize  intersection homology 
and cohomology of intersection spaces of the universal example for $SU(3)$.

In this Appendix we will prove a theorem related to the suspension over a smooth manifold.
\begin{theorem}\label{HI-suspension}
Let $L$ be a smooth, simply connected, oriented manifold of dimension $l$ and let $\overline{p}$ denote an extended perversity. Then:
\[\tilde{H}I^{\overline{p}}_j\bigl ( susp(L)	\bigr)=\begin{cases}
\tilde{H}_{j-1}(L), & 0<i<l-\overline{p}(l+1)\\
\tilde{H}_j(L)\oplus \tilde{H}_{j-1}(L), & j=l-\overline{p}(l+1)\\
\tilde{H}_j(L), &  \rm{otherwise},
\end{cases}\]
where by $\rm{susp}(L)$ we mean the suspension over $L$.
\end{theorem}
\begin{remark}
Note that the suspension of a smooth manifold $M$ is not normally a smooth
manifold itself (the suspension of $M$ is only a smooth manifold only if
$M$ is a sphere), so we would not expect  the ordinary cohomology of the suspension of $M$ to 
satisfy Poincar\'e duality.
\end{remark}

Let $L$ be a smooth, simply connected  manifold of dimension $L$. By $X:={\rm susp}(L)$ we mean the quotient space obtained from $L\times[-1,1]$ 
by collapsing $L\times \{1\}$ to one point (denote this point by $v$), and $L\times \{-1\}$ to another point (denoted by $u$). We observe that $(X,\{u,v\})$ 
is a depth 1 Thom-Mather pseudomanifold  with trivial link bundle
\[L\times \{u,v\}\xrightarrow{pr_2} \{u,v\}.\]
Following the construction given in subsection \ref{theorem 1.1}, we have $\overline{M}=CT(X)\cong [-1,1]\times L$. 

Fix a perversity $\overline{p}$ and set $k:=l-\overline{p}(l+1)$, then
\begin{align}
I^{\overline{p}}X=L\times[-1,1]\cup_g c(L_{<k}\times \{u,v\}),
\end{align}
where the map $g$ is defined to be the composition
\begin{align}
L_{<k}\times\{u,v\}\rightarrow L\times \{u,v\}\xrightarrow{\cong} L\times\{-1\}\sqcup L\times \{1\}\hookrightarrow L\times[-1,1]
\end{align}
where $f:L_{<k}\rightarrow L$ is a stage $k$ Moore approximation of $L$. 

\subsection{Proof of Theorem \ref{HI-suspension}}

Throughout this section, $X={\rm susp}(L)$ where $L$ is a smooth manifold satisfying the conditions given in Theorem \ref{HI-suspension}. Moreover,
 we set $k= l-\overline{p}(l+1)$.

First we are going to prove the following lemma:
\begin{lemma}\label{higher-dim-homology}
For $i>k$ we have
\[\tilde{H}I^{\overline{p}}_i(X)=\tilde{H}_i(L).\]
\end{lemma}
\begin{proof}
The proof of this lemma is very similar to the proof of Theorem \ref{main-result}. Define two open sets $A$ and $B$ as follows.
The space $A$ is defined by 
\[A=L\times[-1,1]\sqcup L_{<k}\times\{u,v\}\times[0,1/2+\epsilon)/\sim\]
where the equivalence relation is given by:
\begin{align*}
&(l,u,0)\sim (f(l),-1)\\
&(l,v,0)\sim (f(l),1). 
\end{align*}
On the other hand, the space $B$ is defined by 
\[B=c(L_{<k}\times\{u,v\})\setminus L_{<k}\times \{u,v\}\times [0,1/2-\epsilon) .\]
By similar reasoning as in the proof of Theorem \ref{main-result}, we see that $A$ deformation retracts to $L$, $B$ is contractible and $A\cap B$ 
is homotopy equivalent to 
$L_{<k}\sqcup L_{<k}.$ 

For $j>k$ the Mayer-Vietoris sequence gives 
\[0 \to \tilde{H}_j(L) \to  \tilde{H}_j(I^{ \bar{p}} (X))  \to 0\]
as $\tilde{H}(L_{<k})\oplus \tilde{H}(L_{<k}) =0$ for $j\geq k$. 
\end{proof}
 
We now require the following lemma:
\begin{lemma}
For $i<k$
\[\tilde{H}_i(I^{\overline{p}}X)=\tilde{H}_{i-1}(L).\]
\end{lemma}
\begin{proof}
This time we cover $I^{\overline{p}}X$ with two different open sets $C,D$ as follows:
\begin{align*}
&C=L\sqcup_g c(L_{<k}\times \{u,v\})\setminus L_{<k}\times\{u\}\times[1/2-\epsilon,1/2+\epsilon,]\\
&D=L\sqcup_g c(L_{<k}\times \{u,v\})\setminus L_{<k}\times\{v\}\times[1/2-\epsilon,1/2+\epsilon.]
\end{align*}
Now we have that $C$ and $D$ deformation retract to ${\rm Cone}(f)$. Moreover, $C\cap D$ is 
homotopy equivalent to $L$. For $i<k$ the Mayer-Vietoris sequence with respect to the cover $\{C,D\}$ gives
\[ 0 \to \tilde{H}_j(I^{\overline{p}}X) \to  \tilde{H}_{j-1}(L)  \to 0\]
as $\tilde{H}_j({\rm Cone} (f))\oplus \tilde{H}_j({\rm Cone}(f))=0$ for $j<k$.
\end{proof}

\noindent{\em Conclusion of proof of Theorem \ref{HI-suspension}:}
In order to complete the proof of Theorem \ref{HI-suspension}, we need to calculate $H_k(I^{\overline{p}}X)$.
Once again we consider the Mayer-Vietoris sequence with respect to the cover $I^{\overline{p}}X=A\cup B$ given in the proof of Lemma
\ref{higher-dim-homology}.

 \begin{equation}\label{M-V}
\begin{diagram}  \dots &\rTo& 0 = \tilde{H}_k(L_{<k})\oplus\tilde{H}_k(L_{<k}) &\rTo^{\gamma_k}& \tilde{H}_k(L) &\rTo^{\alpha_k}&  \tilde{H}_k (I^{\bar{p}}(X)) &\rTo^{\beta_k}& \cr  
\dots &\rTo& \tilde{H}_{k-1} (L_{<k})\oplus\tilde{H}_{k-1}(L_{<k})  &\rTo^{\gamma_{k-1}}& \tilde{H}_{k-1} (L) &\rTo^{\alpha_{k-1}}&
  \tilde{H}_{k-1} (I^{\bar{p}}(X) ) &\rTo^{\beta_{k-1}}&\dots \end{diagram} \end{equation}

Using this diagram, we observe that $\a_k$ is injective, as the top line of the diagram gives
\[0\xrightarrow{\gamma_k} \tilde{H}_k(L)\xrightarrow{\a_k}\tilde{H}_k(I^{\overline{p}}X)\rightarrow\dots \]

The map $\gamma_i$ is given by
\begin{align*}
\gamma_i:&\tilde{H}_i(L_{<k})\oplus\tilde{H}_i(L_{<k})\rightarrow \tilde{H}_i(L),\quad (\omega,\eta)\mapsto f_*(\omega)+f_*(\eta).
\end{align*}
For $i<k$, $f_*:\tilde{H}_i(L_{<k})\longrightarrow \tilde{H}_i(L)$ is an isomorphism. This implies that $\gamma_i$ is a surjective map with 
\[\ker(\gamma_i)=\{(\omega,-\omega)\vert\quad \omega \in \tilde{H}_i(L_{<k})\}\cong \tilde{H}_i(L_{<k})=\tilde{H}_i(L).\]
 In particular we get $\ker(\gamma_{k-1})\cong\tilde{H}_{k-1}(L).$ 
Now we can calculate $\tilde{H}_k(I^{\overline{p}}X)$.
\[\tilde{H}_k(I^{\overline{p}}X)={\rm Im}(\beta_k)\oplus \ker(\beta_k)=\ker(\gamma_{k-1})\oplus 
{\rm Im}(\a_k)=\tilde{H}_{k-1}(L)\oplus \tilde{H}_{k}(L).\]
\hfill $\square$

This completes the proof of Theorem \ref{HI-suspension}.

\end{document}